\documentclass[12pt]{amsart}
\usepackage{amsfonts, amstext, amsmath, amsthm, amscd, amssymb, float, longtable, fullpage}
\usepackage{epsfig, graphics, psfrag}
\usepackage[hidelinks]{hyperref}
\usepackage{graphicx}
\usepackage{color}
\usepackage[font=small, labelfont=bf]{caption}
\usepackage{float}

\theoremstyle{plain} 
\newtheorem{Theorem}{Theorem}[section]
\newtheorem{Corollary}[Theorem]{Corollary}

\theoremstyle{definition} 
\newtheorem{Definition}[Theorem]{Definition}

\newcommand{\RR}{\mathbb{R}}
\newcommand{\ZZ}{\mathbb{Z}}

\usepackage{titlesec}
\titleformat{\section} {\normalfont\scshape\bfseries\filcenter}{\thesection }{1em}{}
\titleformat{\subsection} {\normalfont\scshape\bfseries\filcenter}{\thesubsection}{1em}{}

\newtheorem{innercustomthm}{Theorem}
\newenvironment{customthm}[1]
{\renewcommand\theinnercustomthm{#1}\innercustomthm}
{\endinnercustomthm}

\title{Legendrian Knots and Multi-Crossings}
\author{Amit Kumar, Jake Murphy, Brian Naff}

\begin{document}

\begin{abstract}

It was shown in~\cite{adams2015knot} that any smooth knot can be represented by an \"ubercrossing projection, i.e. a knot projection with no crossings aside from a single multi-crossing. We extend this idea to Legendrian knots and investigate \"ubercrossing and petal projections in the front and Lagrangian projections. We show that any Legendrian knot with an \"ubercrossing projection in the front projection is smoothly isotopic to the unknot and we demonstrate how to compute the $tb$ and rotation numbers for petal projections in the Lagrangian projection.
    
\end{abstract}

\maketitle

\section{Introduction}

Recently, the authors of \cite{adams2015knot} created an algorithm which, given any smooth knot diagram, produces a projection of the knot with a single multi-crossing. This projection is called an \textit{\"ubercrossing projection}. As opposed to classical knot diagrams, where crossings are points where exactly twos strands cross, a \textit{multi-crossing} is a point in the projection where two or more strands cross. We call a multi-crossing where $n$ strands cross an $n$-\textit{crossing}. They also showed that these \"ubercrossing projections can be refined into an \"ubercrossing projection without any nested loops, called a petal projection. They used these projections to created two new knot invariants using these projections, namely the \textit{\"ubecrossing number} $\ddot{u}(K)$ and the \textit{petal number} $p(K)$. These are the minimum number of loops in any \"ubercrossing or petal projection of $K$, respectively. 

In \cite{MR3334663}, \cite{MR4523300}, \cite{MR4462153}, and \cite{MR3868230}, bounds for these invariants were given in terms of well-known invariants, such as the crossing number and knot determinant. Values for the petal number of certain classes of torus knots were computed in \cite{MR4523300} and \cite{MR4641892}. The authors of \cite{MR4012239} created Reidemeister-like moves which determine when two petal projections are equivalent. The Adams, Hoste, and Palmer did this for knot projections where every crossing is a 3-crossing in \cite{MR4038326}. Adams and Kehne use multi-crossing projections to give upper bounds on the hyperbolic volume of a link in \cite{MR4124135} and the authors of \cite{MR3733244} investigate the Kauffman bracket for these projections. Finally, the idea of petal projections was extended to virtual links and it was shown that any virtual link can be represented by a petal projection in \cite{MR4641887}.

The goal of this paper is to apply the concepts of \"ubercrossing and petal projections to Legendrain knots. One major difference between \"ubercrossing projections of smooth knots and of Legendrian knots is that non-isotopic smooth knots can have the same \"ubercrossing projection, while a Legendrian \"ubercrossing projection in the front of Lagrangian projection will uniquely determine a Legendrian knot, since we are able to recover the $y$-coordinates of a knot in the front projection and the relative $z$-coordinates in the Lagrangian projection. We show that any \"ubercrossing projection of a Legendrian knot in the front projection is smoothly isotopic to the unknot and that any smooth knot. For the front projection, we prove the following:

\begin{customthm}{A}
\textit{If $K$ is a Legendrian knot with an \"ubercrossing projection in the front projection, then $K$ is smoothly isotopic to the unknot.}
\end{customthm}

We introduce the idea of a petal projection in the Lagrangian projection, called a Lagrangian petal projection, and have shown the following:

\begin{customthm}{B}
\textit{For  a smooth knot K there exists a Legendrian knot $\Lambda$ with smooth knot type $K$ such that $\Lambda$ has a Lagrangian petal projection. }   
\end{customthm}

We provide formulas for computing the $tb$ and the rotation number for these Lagrangian petal projections. Finally, we demonstrate a family of Legendrian knots with petal projections in the Lagrangian projection with strictly increasing $tb$.

\subsection{Acknowledgements}

The authors would like to thank Shea Vela-Vick and Angela Wu for their help in this project and their advice in writing this paper. This material is based upon work supported by the National Science Foundation under Award No.~1812061 and Award No.~2231492.
\section{Background}

In this section, we introduce multi-crossing knot projections, \"ubercrossing projections, and petal projections. Then, we will define Legendrian knots and cover some of their classical invariants, namely the rotation number and Thurston-Bennequin number.

\subsection{Multi-crossing Projections}

The idea of a multi-crossing was introduced in \cite{adams2013triple} to generalize typical 2-strand crossings in knot projections.

\begin{Definition}

    An $n$-\textit{multi-crossing} of a knot projection is a point where $n$ strands intersect such that each strand bisects the crossing.
    
\end{Definition}

These multi-crossings were used to define \"ubercrossing projections and petal projections in \cite{adams2015knot}.
\begin{Definition}

An \textit{\"ubercrossing projection} is a knot projection with a single multi-crossing. A \textit{petal projection} is an \"ubercrossing projection which does not contain any nested loops.
    
\end{Definition}

\begin{Theorem}[\cite{adams2015knot}]
    Every smooth knot has an \"ubercrossing projection.
\end{Theorem}

Any \"ubercrossing projection can be isotoped into a petal projection, which gives the following corollary.

\begin{Corollary}[\cite{adams2015knot}]
    Every knot has a petal projection.
\end{Corollary}

\subsection{Legendrian Knots} We will give a brief introduction to Legendrian knots. Interested readers should consult~\cite{etnyre2005legendrian} or~\cite{geiges2008introduction} for a comprehensive introduction to the topic.

\begin{Definition}

 The \textit{standard contact structure} on $\mathbb{R}^3$ with coordinates $(x,y,z)$, denoted $(\mathbb{R}^3, \xi_{std})$, is the plane field defined by $\xi_{std} = \text{ker}(dz - ydx)$. 
    
\end{Definition}

Figure~\ref{fig:my_label} shows the standard contact structure on $\mathbb{R}^3$. We use the standard contact structure to define Legendrian knots.
\vspace{2mm}
\begin{figure}
    \centering
    \includegraphics[scale = 0.5]{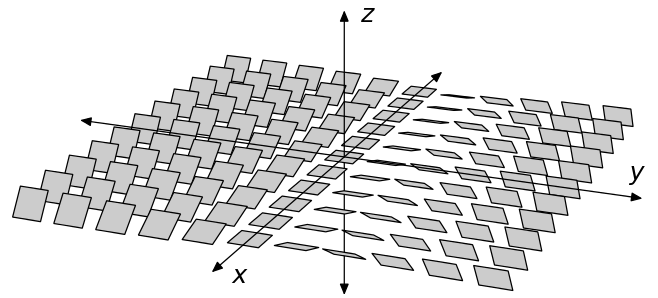}
    \caption{The standard contact structure, $(\mathbb{R}^3, \xi_{std})$. Source:~\cite{pict}.}
    \label{fig:my_label}
\end{figure}

\begin{Definition}
A \textit{Legendrian knot} $\Lambda$ is a knot, such that the tangent vector at each point $p$ of $\Lambda$ is contained in the plane of $\xi_{\text{std}}$ at $p$.
\end{Definition}

Note that if $\Lambda: S^1 \to \mathbb{R}^3$ is a Legendrian knot in $(\mathbb{R}^3, \xi_{std})$ given by $$\Lambda: t \mapsto (x(t), y(t), z(t)),$$ then our previous conditions imply
$$z'(t)-y(t)x'(t) = 0. $$

\begin{Definition}
    
Two Legendrian knots $\Lambda_0$ and $\Lambda_1$ are called \textit{Legendrian isotopic} if there is a smooth map $\kappa : S^1 \times [0, 1] \to \mathbb{R}^3$ such that $\kappa(S^1,0) = \Lambda_0$, $\kappa(S^1,1) = \Lambda_1$, and $\kappa(S^1,s)$ is a Legendrian knot for every $s \in [0,1]$.

\end{Definition}

We typically consider two types of projections of Legendrian knots, namely the projection to the $xz$-plane, called the \textit{front projection}, and projection to the $xy$-plane, called the \textit{Lagrangian projection}.

For the front projection, let $\Pi$ be the projection onto the $xz$-plane. For a Legendrian knot $\Lambda$, we can view the front projection of $\Lambda$ as a map $\Pi(\Lambda):S^1\to\RR^2$ which sends $t\in S^1$ to $(x(t),z(t))$. Since $\Lambda$ satisfies the equation $z'(t) -y(t)x'(t) = 0 $, we can recover the $y$-coordinate of a point in the front projection by looking at the slope of that point, since $$y(t) = \frac{z'(t)}{x'(t)}.$$

We can similarly recover information about the $z$-coordinate of a Legendrian knot $\Lambda$ from the Lagrangian projection by using the equation $z'(t)- y(t)x'(t) = 0$. To accomplish this, we integrate $y(t)x'(t)$ along the Lagrangian projection of $\Lambda$ to get \[z(t_2)-z(t_1)=\int_{t_1}^{t_2} y(t)x'(t)dt.\]
Note that this allows us to recover the $z$-coordinate up to translations of $\Lambda$ along the $z$-axis.

Lastly, we will cover two invariants of Legendrian knots and how to compute them in the Lagrangian projection, namely the rotation number and the Thurston-Bennequin number.

\begin{Definition}
    For a Legendrian knot $\Lambda$ in the Lagrangian projection, the \textit{Thurston-Bennequin number} of $\Lambda$ is given by $tb(\Lambda)=\text{writhe}(\Lambda)$. The \textit{rotation number} of $\Lambda$ is the winding number of $\Lambda$ in the Lagrangian projection.
\end{Definition}

\section{\"Ubercrossing Projections in the Front Projection}

In this section, we first investigate which Legendrian knots can be presented with an \"ubercrossing projection in the front projection.

\begin{Theorem}
If $K$ is a Legendrian knot with an \"ubercrossing projection in the front projection, then $K$ is smoothly isotopic to the unknot.
\end{Theorem}

\begin{proof}

Since the slope of a curve in the front projection of a Legendrian link determines the $y$-coordinate, any \"ubercrossing projection in the front projection uniquely determines a Legendrian link. We will first prove by induction on $n$ that any $n$-strand multi-crossing can be ``spread out'' into a half-twist of $n$-strands.

This is clearly true for our base case of $n=2$, as this is regular crossing. Suppose a $k$-strand multi-crossing can be isotoped to a half-twist on $k$ strands. We label these strands from top to bottom on the left by $1,\ldots,k$ as in Figure~\ref{sl}. Then, if we take a half-twist on $k+1$ strands, we can isotope the twist of the first $k$ strands into a single multi-crossing, as shown in Figure~\ref{ht}. Then we can isotope the final strand to give a multi-crossing of $k+1$ strands.

\begin{figure}
    \centering
    \includegraphics[width=1.5in]{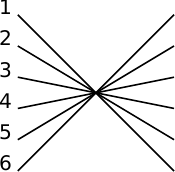}
    \caption{Labeling of a 6-strand multi-crossing.}
    \label{sl}
\end{figure}

\begin{figure}
    \centering
    \includegraphics[width=5.5in]{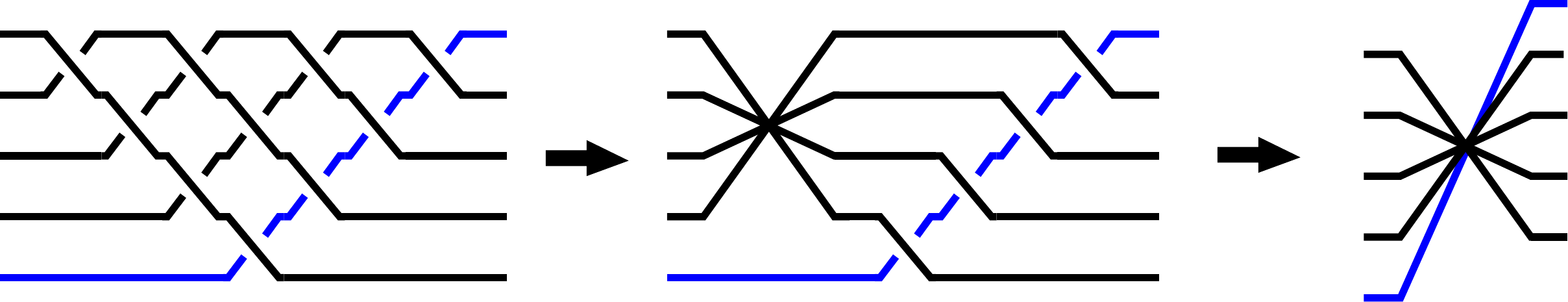}
    \caption{An isotopy showing that if a half-twist on 4 strands can be isotoped to a single multi-crossing, then so can a half-twist on 5 strands.}
    \label{ht}
\end{figure}

We now show that any way we close a single $n$-strand half-twist (without adding additional crossings) results in an unknot or a multiple component link. First, suppose we close the half-twist using a rainbow closure. Clearly, if $n=1$ or $n=2$, then the result is smoothly isotopic to the unknot. Note that a half twist must send the top strand to the bottom strand and vice versa, meaning the top and bottom strand of the half twist get closed into a single component in a rainbow closure. Therefore, if $n>2$ the rainbow closure has multiple components.

Finally, if our closure is not a rainbow closure, then there must be two adjacent strands of the half twist which are joined together. In this case, we can isotope the knot to a closure of a half twist on 2 fewer strands, possibly with an extra unlinked unknot, as shown in Figure~\ref{rr}.
\end{proof}
\begin{figure}
    \centering
    \includegraphics[width=4.5in]{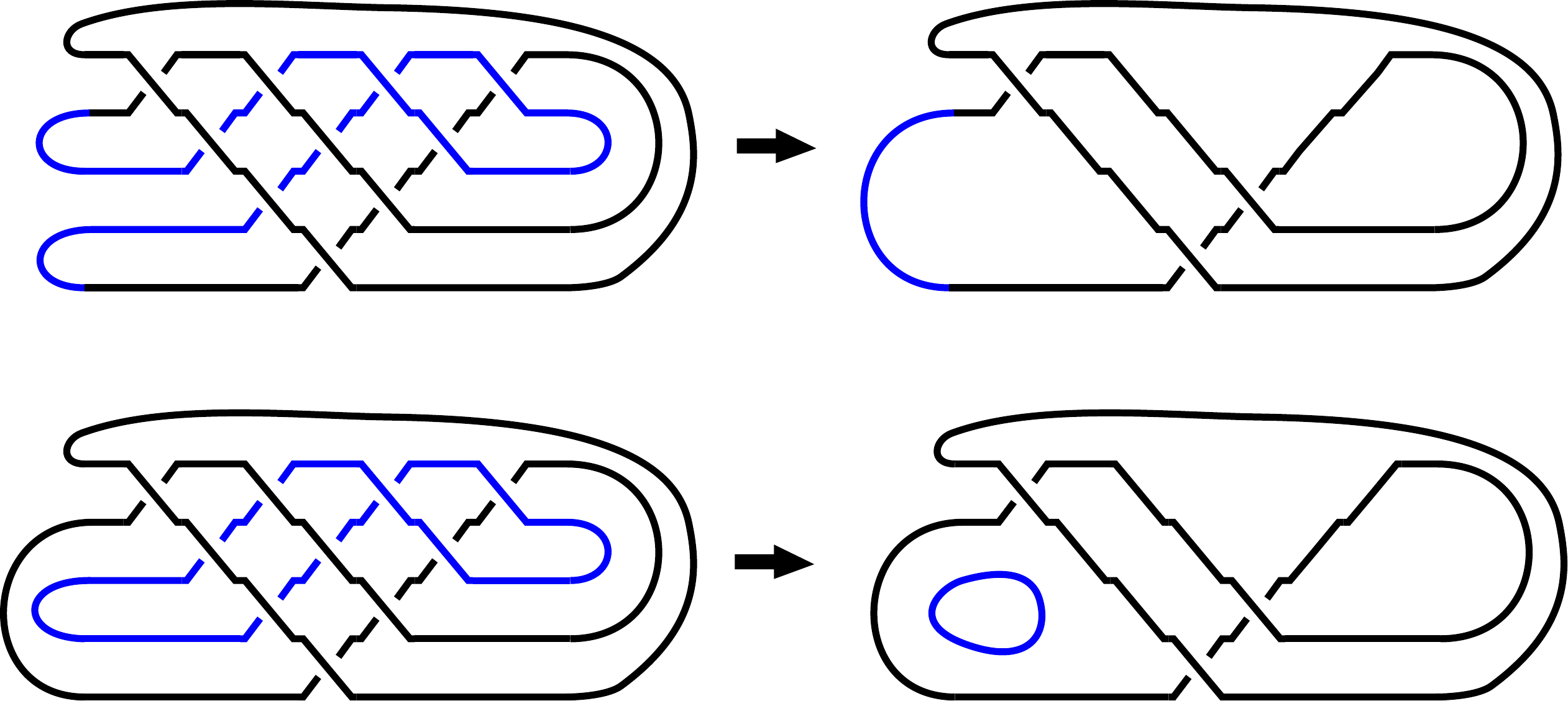}
    \caption{The top is an example of when a non-rainbow closure of a half twist can be isotoped to a closure of a half-twist on fewer stands. The bottom shows an example of when a non-rainbow closure of a half-twist can be isotoped to a unlinked union of an unknot and a closure of a half-twist on fewer strands.}
    \label{rr}
\end{figure}

\section{Lagrangian Petal Projections}

In this section, we define petal projections for Legendrian knots in the Lagrangian projection and use these to compute the rotation and Thurston-Bennequin numbers.

\begin{Definition}

A \textit{Lagrangian petal projection} of a Legendrian knot is a petal projection in the Lagrangian projection with possible half-twists on the petals, such as in Figure~\ref{petal}. Note that if the petal projection of a knot has at least two petals, then there must be an odd number of petals, as otherwise the projection would have multiple components.

\end{Definition}

Since we are able to recover the relative $z$-coordinates of points of our knot from the Lagrangian projection, a Lagrangian petal projection is a projection of a unique knot. Therefore, it is often useful to refer to these petal projections by the order of heights of the strands in the multi-crossing.

\begin{Definition}\label{height}

For an $n$ strand multi crossing, we label the strands by $1,2,\ldots,n$ as shown in Figure~\ref{sl}. We define a height function $h:\{1,2,\ldots,n\}\to\{1,2,\ldots,n\}$, where $h(i)=k$ if the $i$th strand is the $k$th highest strand at the multi-crossing. We refer to a given Lagrangian petal projection with height function $h$ as an $(h(1),\ldots,h(n))$-petal projection.

\end{Definition}

\begin{figure}
    \centering
    \includegraphics[width=1.4in]{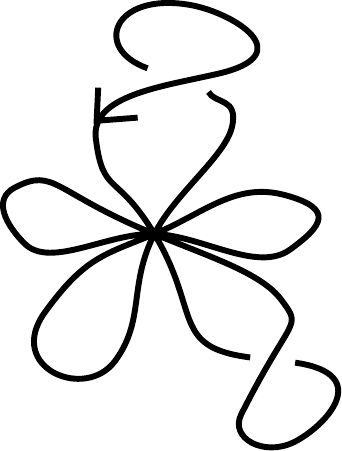}
    \caption{An example of a Lagrangian petal projection.}
    \label{petal}
\end{figure}

\begin{Theorem}
For  a smooth knot K there exists a Legendrian knot $\Lambda$ smoothly isotopic to $K$ such that $\Lambda$ has a Lagrangian petal projection.
\end{Theorem}

\begin{proof}

For a smooth knot $K$ there exists a petal projection of $K$. Suppose the petal projection has an $n$-multi-crossing. Let $h$ be the height function of this multi-crossing defined as in Definition~\ref{height}. Then any $(h(1),\ldots,h(n))$-Lagrangian petal projection is a projection of a Legendrian knot which is smoothly isotopic to $K$.

We now show how to construct this $(h(1),\ldots,h(n))$-Lagrangian petal projection from the petal projection of $K$. Let the signed distance between two strands $i$ and $i+1$ be given by $d= h(i+1)- h(i)$. If $h(i) < h(i+1)$, isotope the petal connecting strands $i$ and $i+1$ to have an area of 1 and apply an R1 move to the petal such that a new negative crossing is created and the new disk has area $1+d$, as shown in Figure~\ref{AddTwist}. If $h(i) > h(i+1)$, we isotope the petal connecting strands $i$ and $i+1$ so that the petal bounds an area of $d$. We also do this for the petal containing strands $1$ and $n$.

This process results in an $(h(1),\ldots,h(n))$-Lagrangian petal projection and does not change the smooth knot type of $K$.
\begin{figure}
    \centering
    \includegraphics[width=3.6in]{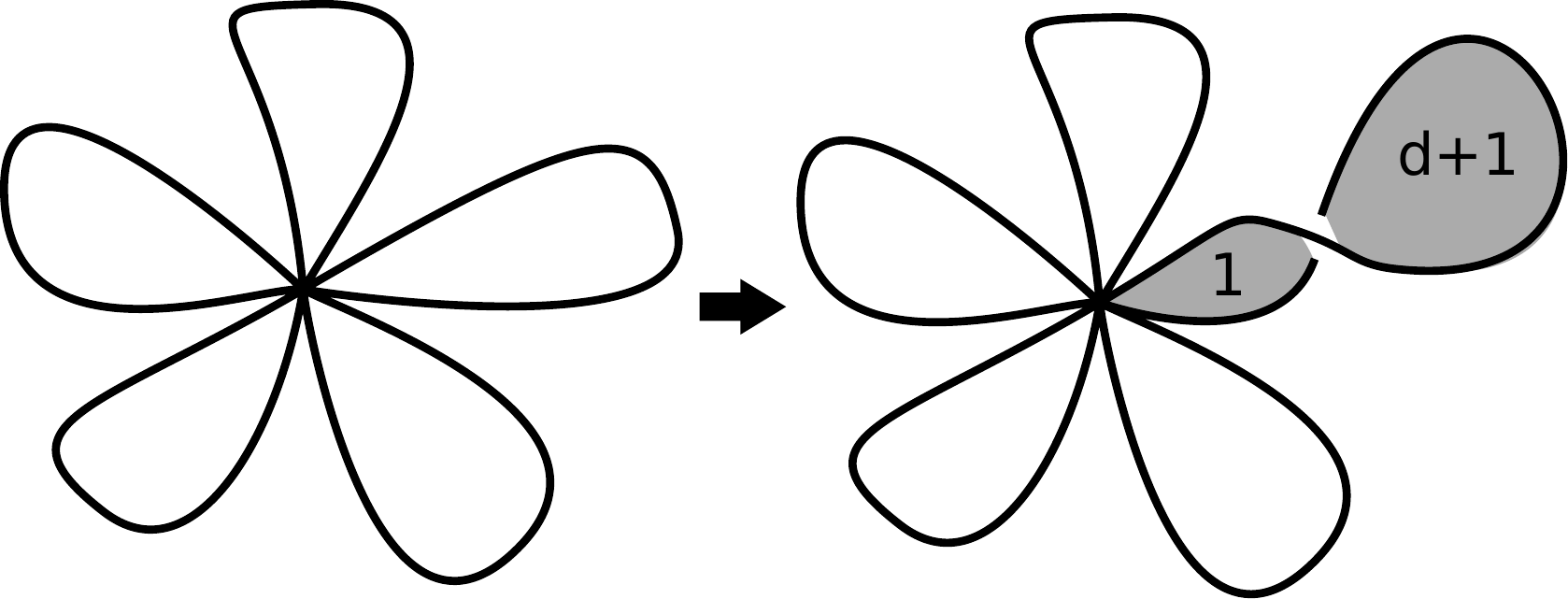}
    \caption{Applying an R1 move to a petal projection.}
    \label{AddTwist}
\end{figure}
\end{proof}

Next, we investigate how stabilizations effect Lagrangian petal projections.

\begin{Theorem}

If a Legendrian knot $\Lambda$ has a Legendrian petal projection, then any Legendrian knot $\Lambda'$ obtained from $\Lambda$ via stabilization also has a Legendrian petal projection.
    
\end{Theorem}

\begin{proof}
We can isotope any stabilization of $\Lambda$ away from the multi-crossing. If we stabilize $\Lambda$ by applying an R1-move to a strand in the boundary of a region $R$ with area $A$, we can we can isotope the new region to have an area of 1 and $R$ to have an area of $A+1$. This new diagram is a valid Lagrangian petal projection for the stabilization of $\Lambda$ as it preserves the heights of the strands at the multi-crossing.
\end{proof}

Now we show how to compute the rotation number and $tb$ from Lagrangian petal projections. We say that an $(h(1),\ldots,h(n))$-Lagrangian petal projection is \textit{standard} if each petal has at most 1 twist and no bounded region is contained within another bounded region.

\begin{Theorem}

A Legendrian knot $\Lambda$ with a standard Lagrangian petal projection with $n$ petals and $k$ half twists oriented as in Figure~\ref{petal} has rotation number $\frac{n+1}{2}-k$.
    
\end{Theorem}

\begin{proof}

The rotation number of a Legendrian knot in the Lagrangian projection is given by the winding number. The winding number of a petal projection with $n$ petals and no half-twists is $\frac{n+1}{2}$. Each half-twist on the outside of a petal will decrease the rotation number by 1. Therefore, the rotation number of the knot is $\frac{n+1}{2}-k$.
\end{proof}

Recall that computing the $tb$ is equivalent to computing the writhe from the Lagrangian projection. We define the following function to help compute the writhe from a Lagrangian petal projection.

Let $f:\{1,2,\ldots,n\}\times\{1,2,\ldots,n\}\to\ZZ_2$ be defined by $f(i,j)=i+j\ (\text{mod 2})$. Recall the height function $h$ defined in Definition~\ref{height}. We define a sign function $\sigma:\{i,j\mid 1\leq i<j\leq n\}\to\{-1,1\}$ by 

\begin{align}
\sigma(i,j) = \left\{ \begin{array}{cc} 
                1 & \hspace{5mm} h(i)>h(j)\ \text{and}\ f(i,j)=0,   \\
                1 & \hspace{5mm} h(i)<h(j)\ \text{and}\ f(i,j)=1,   \\
                -1 & \hspace{5mm} h(i)>h(j)\ \text{and}\ f(i,j)=1, \\
                -1 & \hspace{5mm} h(i)<h(j)\ \text{and}\ f(i,j)=0. \\
                \end{array} \right.
\end{align}

Any two strands in the multi-crossing cross each other exactly once. Therefore, when we spread out the multi-crossing to a typical link diagram, we will have exactly one crossing for each pair of strands. This sign function $\sigma$ computes the sign of each of these crossings. This gives us our next theorem.

\begin{Theorem}

For a Legendrian knot $\Lambda$ with a Lagrangian petal projection with $k$ half-twists and an $n$-multicrossing, the $tb$ of $\Lambda$ is given by

\[tb(\Lambda)=-k+\sum_{1\leq i<j\leq n}\sigma(i,j).\]
\end{Theorem}

\begin{proof}
    When computing the writhe of $\lambda$ from the Lagrangian petal projection, each of the $k$ half twist contribute -1 to the total. If we ``spead" out the $n$-multi-crossing, strands $i$ and $j$ cross each other exactly once and this crossing contributes $\sigma(i,j)$ to the writhe.
\end{proof}

Let $\Lambda_n$ be the Legendrian knot with a standard $(1,\frac{n+3}{2},2,\frac{n+5}{2},\ldots,\frac{n-1}{2},n,\frac{n+1}{2})$-Lagrangian petal projection. This petal projection has exactly $\frac{n-1}{2}$ half twists since $h(i)>h(i+1)$ if and only if $i$ is even. In fact, if $i$ is even and $i<j$, then $h(i)>h(j)$ if and only if $j$ is also even. Therefore, for a fixed even $i$, we have \[\sum_{i<j\leq n} \sigma(i,j)=n-i.\] If $i$ is odd and $i<j$, then $h(i)<h(j)$, so for a fixed odd $i$, we have \[\sum_{i<j\leq n}\sigma(i,j)=0.\] Therefore, the knot given by this projection has \[tb=\left(\sum_{0\leq k\leq \frac{n-1}{2}} 2k+1 \right)-\frac{n-1}{2}=\left(\frac{n-1}{2}\right)^2-\frac{n-1}{2}.\]
This is a sequence of knots where $tb$ strictly increases as $n$ increases, which shows that we can find a Lagrangian petal projection for any given $tb$. We believe these knots to have the largest possible $tb$ for a Lagrangian petal projection with $n$ petals.

\begin{Theorem}\label{upperbound}

Any Legendrian knot $\Lambda$ with a Lagrangian petal projection with $n$ petals will have $tb(\Lambda)\leq\left(\frac{n-1}{2}\right)^2-1$.

\end{Theorem}

\begin{proof}

To prove this, we will show that the largest $\sum_{1\leq i<j\leq n}\sigma(i,j)$ can be is $(n-1)^2$. First, note that we can rotate any Lagrangian petal projection so that $h(1)=1$. In this case, $\sigma(1,i)$ alternates between -1 and 1. Therefore, we have that \[\sum_{1\leq i<j\leq n}\sigma(i,j)=\sum_{2\leq i<j\leq n}\sigma(i,j).\]

Next, we show that it is impossible for \[\sigma(i,j)=\sigma(i+1,j)=\sigma(i,j+1)=\sigma(i+1,j+1)=1\] for any $j>i+1$. Suppose the contrary. We further assume that both $i$ and $j$ are even, since other cases follow similarly. If $\sigma(i,j)=\sigma(i+1,j)=1$, then $h(i)>h(j)>h(i+1)$. If $\sigma(i,j+1)=\sigma(i+1,j+1)=1$, then $h(i+1)>h(j+1)>h(i)$, which is a contradiction. Therefore, at least one of $\sigma(i,j)$, $\sigma(i+1,j)$, $\sigma(i,j+1),$ $\sigma(i+1,j+1)$ is negative.

Consider squares containing vertices $(i,j)$, $(i+1,j)$, $(i,j-1)$, and $(i+1,j-1)$ for even $i$ and $j$ with $i<j\leq n$. Each of these squares must have at least one vertex where $\sigma$ has a value of -1. Since the number of these squares is \[\sum_{k=1}^{\frac{n-1}{2}}k=\frac{(n+1)(n-1)}{8},\]
we know that at least $\frac{(n+1)(n-1)}{8}$ of the summands of $\sum_{2\leq i<j\leq n}\sigma(i,j)$ are $-1$. Therefore, we have the inequality
\[ \sum_{2\leq i<j\leq n}\sigma(i,j) \leq \frac{n(n-1)}{2} - 2\frac{(n+1)(n-1)}{8} = \left(\frac{n-1}{2}\right)^2.\]

Since any Lagrangian petal projection must have at least one half twist, we get the desired inequality $tb(\Lambda)\leq\left(\frac{n-1}{2}\right)^2-1$.
\end{proof}

\subsection{Open Problems}

Here are some problems we have considered, but have not answered. First, while any \"ubercrossing projection of a Legendrian knot in the front projection is trivial, is there anything we can say about Legendrian knots in the front projection with exactly 2 multi-crossings? With $n$ multi-crossings? Second, can we improve our upper bound on $tb$ from Theorem~\ref{upperbound}? We believe that the $(1,\frac{n+3}{2},2,\frac{n+5}{2},\ldots,\frac{n-1}{2},n,\frac{n+1}{2})$ Lagrangian petal projection attains the highest $tb$ for a projection with $n$ petals and that it is unique. These knots achieve the largest possible writhe for such a Lagrangian petal projection, so any knot with a larger $tb$ would need fewer than $\frac{n-1}{2}$ half-twists. Finally, is there a Lagrangian petal projection for every max $tb$ Legendrian knot? In other words, does every Legendrian knot have a Lagrangian petal projection?

\bibliography{Bibliography}

\bibliographystyle{alphaurl}

\end{document}